\documentclass[10pt]{amsart}

\usepackage{xcolor}

\usepackage{hyperref}
\hypersetup{colorlinks=true,
urlcolor=blue, citecolor=blue, linkcolor=blue}\usepackage[english]{babel}
\usepackage{amsmath, amssymb, amsthm}
\usepackage{graphicx}
\usepackage{enumitem}

\newif\iftikz
\tikztrue
\iftikz
\usepackage{tikz}
\usetikzlibrary{calc,arrows.meta}
\fi

\newtheorem{thm}{Theorem}[section]
\newtheorem*{yano}{Yano's Conjecture}
\newtheorem*{yanoext}{Extended Yano's Problem}
\newtheorem*{yanoextcj}{Extended Yano's Conjecture}
\newtheorem*{yanoextcj1}{Modified extended Yano's Conjecture}
\newtheorem{prop}[thm]{Proposition}
\newtheorem{cor}[thm]{Corollary}
\newtheorem{lema}[thm]{Lemma}
\theoremstyle{remark}
\theoremstyle{definition}
\newtheorem{ejm}[thm]{Example}

\DeclareMathOperator{\spec}{Spec}
\DeclareMathOperator{\Gr}{Gr}
\newcommand\bn{{\mathbb N}}
\newcommand\bz{{\mathbb Z}}
\newcommand\bq{{\mathbb Q}}
\newcommand\bc{{\mathbb C}}
\newcommand\br{{\mathbb R}}
\newcommand{\cO}{{\mathcal O}}

\def\f{f^{-1}\{0\}}

\title{On the $b$-exponents  of generic isolated plane curve singularities}

\author[E. Artal]{E.~Artal Bartolo${^1}$}

\address{Departamento de Matem\'aticas-IUMA, Universidad de Zaragoza,
c/~Pedro Cerbuna 12, 50009 Zaragoza, SPAIN}

\email{artal@unizar.es}

\thanks{${^1}$Partially supported by the grant 
MTM2016-76868-C2-2-P}

\author[Pi.~Cassou-Nogu\`es]{Pi.~Cassou-Nogu\`es${^2}$}
\address{Institut de Math\'ematiques de Bordeaux, Universit\'e de Bordeaux, 
350, Cours de la Lib\'eration, 33405, Talence Cedex 05, FRANCE}
\email{Pierrette.Cassou-nogues@math.u-bordeaux.fr}
\thanks{${^2}$Partially supported by MTM2016-76868-C2-1-P
 and MTM2016-76868-C2-2-P}

\author[I. Luengo]{I. Luengo${^3}$}

\address{ICMAT (CSIC-UAM-UC3M-UCM), Dpto. de \'Algebra, Geometr{\'i}a y Topolog{\'i}a, 
Universidad Complutense de Madrid,
Plaza de las Ciencias s/n, Ciudad Universitaria, 28040 Madrid, SPAIN}

\email{iluengo@ucm.es}

\author[A. Melle]{A.~Melle-Hern\'andez${^3}$}

\address{Instituto de Matem\'atica Interdisciplinar (IMI), Dpto. de \'Algebra, Geometr{\'i}a y Topolog{\'i}a, , 
Universidad Complutense de Madrid,
Plaza de las Ciencias s/n, Ciudad Universitaria, 28040 Madrid, SPAIN}

\email{amelle@ucm.es}

\thanks{${^3}$Partially
supported by the grant MTM2016-76868-C2-1-P and Grupo Singular}

\dedicatory{Dedicated to the memory of Egbert Brieskorn with great admiration}



\begin{document}

\begin{abstract}
In 1982, Tamaki Yano proposed a conjecture predicting  how is
the set of $b$-exponents of an irreducible  plane curve singularity germ which is generic 
in its equisingularity class. 
In 1986, Pi.~Cassou-Nogu\`es  proved the conjecture for the one Puiseux pair case in \cite{Pi862}.
In \cite{ACLM-Yano2} the authors  proved the conjecture 
for  two Puiseux pairs germs  whose complex algebraic monodromy has distinct eigenvalues. 
A natural problem induced  by Yano's conjecture is, for a generic equisingular deformation of an isolated plane curve singularity germ 
to study how the set of  $b$-exponents  
depends on the topology of the singularity. 
The natural generalization suggested by  Yano's approach holds  in suitable examples 
(for the case of  isolated singularites which are Newton  
non-degenerated, commode and  whose set of spectral numbers are all distincts).
Morevover we show with an example that this natural  generalization is not correct.
 We restrict  to germs  whose complex algebraic monodromy has distinct eigenvalues
such that the embedded resolution graph has vertices of valency at most $3$ and
we discuss some examples with multiple eigenvalues.  
\end{abstract}

\maketitle

\section*{Introduction}

Let $f:(\bc^n,0)\to (\bc,0)$ be  a germ of a complex analytic function  whose  zero locus  $(f^{-1}(0), 0)\subset (\bc^n,0)$ 
 defines  an isolated hypersurface singularity germ, that is  the   Minor number  of $f$ at $0$,
$\mu(f,0):=\dim_\bc\bc\{z_1,\ldots ,z_n\}/\left(\frac{\partial f}{\partial z_1},\ldots, \frac{\partial f}{\partial z_n}\right)$  is finite.
A \emph{Milnor fibration}  was constructed in \cite{Mi68} as follows. 
 Set   $B_{\varepsilon} = \{ z\in  \bc^n: |z| < \epsilon\}$ and
 $S_{\epsilon} = \{ z\in  \bc^n:  |z|= \epsilon\}$, one can choose $\epsilon_0$ such that for all 
$ 0 < {\epsilon} \leq  {\epsilon}_0$,  $f^{-1} (0)$ is transverse to $S_{\epsilon}.$  
For  $0<\eta\ll \epsilon_0$ and $D_{\eta} = \{t\in \bc: |t| < \eta\}$, let $X(t) = f^{- 1} ( t ) \cap B_{\epsilon_0/2}$ 
and $X = f^{- 1} (D_\eta) \cap B_{\epsilon_0/2}$ . By Milnor, for such suitable $\epsilon$ and $\eta$, the mapping
$X\setminus f^{-1} (0)\to D_{\eta}\setminus \{0\} $ is a $C^{\infty}$-locally trivial fibration whose general fibre $F_{f,0}$, 
called \emph{Milnor fibre},
has the homotopy type of a bouquet of exactly  $\mu(f,0)$ of $(n-1)$-dimensional spheres. 

The geometric monodromy   $h_{F_{f,0}}:F_{f,0}\to F_{f,0}$ of the Milnor fibration is the monodromy transformation of the Milnor fibration over the loop
$c\exp(2\pi t), t\in [0,1]$ and $c$ small enough.  The geometric monodromy induces  the complex algebraic monodromy 
$h^{a,j}: H^j (F_{f,0},\bc) \to H^{j} (F_{f,0},\bc)$  whose eigenvalues are roots of unity. 
Since the Milnor fibre is a connected bouquet of $(n-1)$-spheres,
the only interesting algebraic monodromy is  $h^{a,n-1}: H^{n-1}(F_{f,0},\bc) \to H^{n-1}(F_{f,0},\bc),$
where $\dim_{\bc} H^{n-1}(F_{f,0},\bc) =\mu(f,0) .$

Let $\cO$ be the ring  of germs of holomorphic functions on $(\bc^n,0)$, 
let $\mathcal{D}$ be the  ring of germs of holomorphic differential operators of finite order with coefficients in $\cO$. 
Let $s$ be an indeterminate commuting with the elements of $\mathcal{D}$ and set 
$\mathcal{D}[s]=\mathcal{D}\otimes_{\mathbb{C}} \mathbb{C}[s].$

Given a holomorphic germ  $f\in \cO$, one considers 
$\cO\left[\frac{1}{f}, s\right]\cdot f^s$ as a free $\cO \left[\frac{1}{f}, s\right]$-module of rank $1$  
with the  natural  $\mathcal{D}[s]$-module structure. Then,   
there exits a non-zero polynomial $B(s)\in \bc[s]$ and some 
differential operator $P=P(x,\frac{\partial}{\partial x},s)\in\mathcal{D}[s]$, 
holomorphic in $x_1,\dots,x_n$ and polynomial in 
$\frac{\partial }{\partial x_1},\dots,\frac{\partial }{\partial x_n}$, which  satisfy   the following functional equation
in $\cO \left[\frac{1}{f}, s\right] f^s$:
\begin{equation}\label{berstein-rel}
P(s,x,D)\cdot f(x)^{s+1}=B(s)\cdot f(x)^s.
\end{equation}
The monic generator $b_{f,0}(s)$ of the ideal  of such polynomials $B(s)$  is called
the \emph{Bernstein-Sato polynomial} (or $b$-function or  Bernstein polynomial) of $f$ at $0$.  
The same result holds if we replace $\cO$
by the ring of polynomials in a field ${\mathbb K}$ of  zero characteristic with the obvious corrections,  see e.g. 
\cite[Section~10, Theorem~3.3]{Co95}.

This result was first obtained for $f$ polynomial by  Bernstein in~\cite{B72} and in general 
by Bj\"{o}rk \cite{B:81}.
One can  prove  that $b_{f,0}(s)$ is divisible by $s+1$, and we also consider the \emph{reduced Bernstein-Sato polynomial}  
$\tilde{b}_{f,0}(s):=\frac{b_{f,0}(s)}{s+1}$.

In the case where $f$  defines an isolated singularity, one can consider the  nowadays called
 \emph{Brieskorn lattice}  
$H_0^{''}:= \Omega^n /df \wedge d \Omega^{n-2}$  introduced by  Brieskorn in \cite{Br70},  and its saturation
$ \tilde{H}_0^{''}=\sum_ {k\geq 0}  (\partial _t t)^k      H_0^{''} $.  
Malgrange \cite{M:75} showed that  
the reduced Bernstein polynomial  
$\tilde{b}_{f,0}(s)$ is the minimal polynomial of the endomorphism $-\partial_t t$ on the vector space 
$F:=\tilde{H}_0^{''}/ \partial_t^{-1}  \tilde{H}_0^{''}$, whose dimension equals  the
 Milnor number  $\mu(f,0)$ of $f$ at $0$. Following Malgrange \cite{M:75},
the set of $b$-\emph{exponents} are  
the $\mu$ roots  $\{\tilde{\beta}_1, \ldots,\tilde{\beta}_\mu \}$  of the characteristic polynomial of the   endomorphism $-\partial_t t$.
Recall also that $\exp(-2i\pi\partial_t t)$ can be identified with the (complex) \emph{algebraic monodromy}
of the corresponding Milnor fibre $F_{f,0}$ of the singularity at the origin.

Kashiwara~\cite{K:76} expressed these ideas using  differential operators and considered
 $\mathcal{M}:=\mathcal{D}[s]f^s/\mathcal{D}[s]f^{s+1}$,
where $s$ defines an endomorphism of $\mathcal{D}(s)f^s$ by multiplication.
This morphism keeps invariant $\tilde{\mathcal{M}}:=(s+1)\mathcal{M}$
and defines a linear endomorphism of $(\Omega^n\otimes_{\mathcal{D}}\tilde{\mathcal{M}})_0$
which is naturally identified with $F$ and under this identification $-\partial_t t$ 
becomes the endomorphism defined by the multiplication by~$s$.

In~\cite{M:75}, Malgrange proved that the set $R_{f,0}$ of roots of the Bernstein-Sato polynomial 
is contained in $\mathbb{Q}_{<0}$, see also Kashiwara~\cite{K:76}, who also restricts the set of candidate roots.
The number $-\alpha_{f,0}:=\max R_{f,0}$ 
is the opposite of the \emph{log canonical threshold} of the singularity and Saito~\cite[Theorem~0.4]{MS94}
proved that 
\begin{equation}\label{eq:saito}
R_{f,0}\subset[\alpha_{f,0}-n,-\alpha_{f,0}].
\end{equation}
Also Saito in \cite{MS:89} showed that the local moduli of $\mu$-constant deformation is determined by the
\emph{Brieskorn lattice} if the $\mu$-constant stratum is smooth, 
as in the case of germs of plane curves where he gave in \cite[p.~30]{MS:89} a more simple  
formula describing the reduced Bernstein-Sato.  
There are many papers devoted to study  Bernstein-Sato polynomial but it would be  
worthwhile to refer to  the  existence of a relative Bernstein-Sato polynomial in \cite{BGM92}, by  Brian\c{c}on  et al., 
and for results on the computation of the
roots of Bernstein-Sato  polynomial for functions with isolated singularity, even if the methods
used in \cite{BGMM89}  are different. In \cite{BMT07}, Brian\c{c}on  et al.  gave a multiple of the Bernstein-Sato
polynomial for any two variables function with isolated singularities. 
Some general properties of  $\mu$-constant deformations are also given by Varchenko in \cite{V80}.

There is another set  which is important too, the set of    \emph{exponents of the monodromy} 
(or spectral numbers, up to the shift by one, in the terminology of
Var\-chenko \cite{V82}).  This notion was first introduced by Steenbrink \cite{St}.

Let  $f: (\bc^{n},0) \longrightarrow (\bc ,0)$ be a germ of a holomorphic function with isolated singularity.
In \cite{St} Steenbrink constructed a mixed Hodge structure on $H^{n-1}(F_{f,0},\bc)$. 
Let 
$$H^{n-1}(F_{f,0},\bc)_{\lambda}=\ker  (T_s-\lambda:H^{n-1}(F_{f,0},\bc)\longrightarrow H^{n-1}(F_{f,0},\bc));$$
where $T_u,T_s$ are, respectively, the unipotent and semi-simple factors of the Jordan decomposition
  of the monodromy $h^{n-1}$. 

The set $\spec(f)$  of spectral numbers are $\mu$ rational numbers 
$$0<\alpha_1\leq \alpha_2\leq \cdots \leq \alpha_{\mu}<n$$
which are defined by the following condition:
\begin{gather*}
\# \{ j:\exp(-2\pi i\alpha_j)=
\lambda, \lfloor\alpha_j\rfloor= n-p-1\}=
\dim _{\mathbb{C}}\Gr_F^p H^{n-1}(F_{f,0},\mathbb{C})_{\lambda},\qquad
\lambda \neq 1\\
\# \{j:\alpha_j=n-p\}=\dim _{\bc}{\text Gr}_F^p H^{n-1}(F_{f,0},\bc)_1.
\end{gather*}
The set $\spec(f)$ of spectral numbers  is symmetric, that is $\alpha_i+\alpha_{\mu-(i-1)}=n$. 
 It is known that this set  is constant under $\mu$-constant deformation of $f$, see \cite{V82}.

As it is well-known, neither the Bernstein-Sato polynomial nor the $b$-exponents are
constant along $\mu$-constant deformation. Given an equisingular type, a generic
set of $b$-exponents or a generic Bernstein-Sato polynomial are expected.
In~\cite{Y82}, Yano proposed a formula (see next section) for the generic $b$-exponents 
for irreducible germs of curves (combined with the Jordan form of the monodromy,
this also yields to a formula for the generic Bernstein polynomial).
This formula was proved for one-Puiseux pair germs by the second named author in \cite{Pi861}
and reproved by M.~Saito in~\cite{MS:89}.

In \cite{ACLM-Yano2}, the conjecture was proved for irreducible 
singularities with two Puiseux pairs and monodromy without multiple eigenvalues.
In this paper, we discuss how to extend the formula for reducible germs of singularities.
There is a natural interpretation of Yano's formula in terms of the resolution
graph of the singularity, see \eqref{conjecture1}. 
We are going to prove in this paper that this formula holds for
singularities with vertices of valency at most~$3$ (and at most two
vertices of valency~$3$) and monodromy without multiple eigenvalues (distinct from~$1$)
(in fact, the correct hypothesis may be distinct exponents of the monodromy, besides~$1$).

The restriction on the number $3$-valency vertices comes from technical reason
but it is most probably avoidable; for example, the second named author proved it
in~\cite{Pi88} for singularities with non-degenerate and commode Newton polygon
(and distinct exponents for the monodromy). The other two conditions seem to
be more important, since we will give examples where it does not hold in
at least two cases: germs
where the vertices have valencies at most~$3$ but there are multiple exponents, 
and germs with vertices with valency greater than~$3$.
We will discuss also other examples and we will introduce the needed results about improper integrals.

\section{Extended Yano's problem}

Let $f:(\bc^{2},0)\to(\bc,0)$ be a germ of a non-zero holomorphic
function such that its zero locus defines an isolated singularity germ. 

\begin{yanoext}[\cite{Y82}]
For a generic equisingular deformation of an isolated plane curve singularity germ $f:(\bc^2,0)\to (\bc,0)$  
and  Milnor number $\mu,$ 
to study how the set of  $b$-exponents  $\{\tilde{\beta}_1\!,\ldots,\!\tilde{\beta}_\mu\}$ 
  depends on the topology of $f$. 
\end{yanoext}

The local Bernstein-Sato polynomial $b_{f,0}(s)$ of a singularity germ is a powerful analytic invariant, but it is, in general, extremely hard
to compute, even in the case of irreducible  plane curve singularities. It is well-known
that the  Bernstein-Sato polynomial varies in families in the (non-singular) $\mu$-constant stratum 
$\Sigma_{\mu(f,0)}$  of $f$ at $0$. 
 Since, for plane curves
this stratum is irreducible, it is conceivable that a \emph{generic} Bernstein-Sato polynomial
exists, i.e.,  the Bernstein-Sato polynomial of a germ~$f$ with the same topology as $f$, depends on~$f$, but there 
is a \emph{generic} Bernstein-Sato polynomial~$b_{\Sigma_{\mu(f,0)}}^{\text{gen}}(s)$: 
for every $\mu$-constant deformation of such an~$f$, there 
is a Zariski dense open set\ $\mathcal{U}$ on which the Bernstein-Sato polynomial
of any germ in~$\mathcal{U}$ equals 
$b_{\Sigma_{\mu(f,0)}}^{\text{gen}}(s)$.

\subsection{The original Yano's conjecture: the irreducible case}
\mbox{}

 Let  $f$ be an irreducible germ of  plane curve. In 1982, Tamaki Yano \cite{Y82} made a conjecture concerning 
the $b$-exponents of such germs.
Let $(n, b_1,b_2,\ldots, b_g)$  be the characteristic sequence of $f$,  see e.g. \cite[Section 3.1]{W04}. 
Recall that this means that $f(x,y)=0$ has as root (say over~$x$) a Puiseux expansion
$$
x=\dots+a_1 y^{\frac{b_1}{n}}+\dots+a_g y^{\frac{b_g}{n}}+\dots
$$
with exactly~$g$ characteristic monomials.
Denote $b_0:=n$ and
define recursively 
$$
e^{(k)}:=
\begin{cases}
n&\text{ if }k=0,\\
\gcd(e^{(k-1)},b_k)&\text{ if }1\leq k\leq g.
\end{cases}
$$
We define the following numbers for $1\leq k\leq g$:
$$
R_k:=\frac{1}{e^{(k)}}\left(b_k e^{(k-1)} +
\sum_{j=0}^{k-2}b_{j+1}\left(e^{(j)}-e^{(j+1)}\right)\right),\qquad
r_k:= \frac{b_k+n}{e^{(k)}}.
$$ 
Note that $R_k$ admits the following recursive formula:
$$
R_k:=
\begin{cases}
n&\text{ if }k=0,\\
\frac{e^{(k-1)}}{e^{(k)}}\left(R_{k-1}+b_k-b_{k-1}\right)&\text{ if }1\leq k\leq g.
\end{cases}
$$
We end with the following definitions ${R'_0}:=n$,  $r'_0 :=2$ and
for $1\leq k\leq g$:
$$
{R'_k}:=\frac{R_k e^{(k)}}{e^{(k-1)}},\quad
r'_k := \left\lfloor  r_k e^{(k)}/e^{(k-1)} \right\rfloor +1. 
$$   
Yano defined the following polynomial with fractional powers in $t$ 
\begin{equation}\label{eq:generating}
R(n,b_1,\ldots,b_g;t):=
t+\sum_{k=1}^g t^{\frac{r_k}{R_k}}\frac{1-t}{1-t^{\frac{1}{R_k}}}
-\sum_{k=0}^g t^{\frac{r'_k}{R'_k}}\frac{1-t}{1-t^{\frac{1}{{R'_k}}}},
\end{equation}
and he  proved that  $R(n,b_1,\ldots,b_g;t)$ has non-negative coefficients.

\begin{yano}[\cite{Y82}] \hspace*{-2mm}  For almost all irreducible plane curve singularity 
germ $f:(\bc^2,0)\to (\bc,0)$  with characteristic sequence
$(n, b_1,b_2,\ldots, b_g)$, the $b$-exponents  
$\{\tilde{\beta}_1\!,\ldots,\!\tilde{\beta}_\mu\}$ 
 are given by  the generating series
\begin{equation*}
 \sum_{i=1}^{\mu} t^{\tilde{\beta}_i}=R(n,b_1,\ldots,b_g;t).
\end{equation*}
For almost all means for an open dense subset in the $\mu$-constant strata in a deformation space.  
\end{yano}  

Yano's conjecture holds for $g=1$ as it was proved
by Pi. Cassou-Nogu\`es in~\cite{Pi861} making explicitly a relation between 
two  variables improper integrals and the 
Bernstein-Sato polynomial of $f$, see also \cite{Pi862}.

In  \cite{ACLM-Yano2}, the authors,   with the same ideas, 
 were interested in the case $g=2$.
For $g=2$, the characteristic sequence $(n,b_1,b_2)$ can be written as 
$(n_1n_2, mn_2, mn_2+q)$ where $n_1,m,n_2,q\in\mathbb{Z}_{>0}$ satisfying
$$
\gcd(n_1,m)=\gcd(n_2,q)=1.
$$
In  \cite{ACLM-Yano2} we solve Yano's conjecture 
for the case
\begin{equation}\label{eq:simple_roots}
\gcd(q,n_1)=1\text{ or }\gcd(q,m)=1.
\end{equation}
The above condition is equivalent to ask for the algebraic monodromy to have distinct eigenvalues. 
In that case, the~$\mu$ $b$-exponents are all distinct and they coincide with
the opposite of roots of the reduced Bernstein-Sato polynomial (which turns out to be of degree~$\mu$).

To encode the topology of  a germ of an irreducible plane   curve singularity
$(C=\f,0)\subset (\bc^2,0)$ several sets of invariants can be used: Puiseux characteristic exponents, Puiseux  pairs, Newton pairs,
(minimal) embedded resolution graph,  Eisenbud-Neumann splice diagram, semigroup
$\Gamma_{(C,0)}\subset \bn$ generated
by all the possible intersection multiplicities $i(\{h=0\},C)$ at
$0$ for all $h\in {\mathcal O}_{(\bc^2,0)}$, etc.

Let $f:(\bc^{2},0)\to(\bc,0)$ be a germ of a non-zero holomorphic
function $f$. Let $B$ be an open ball centered at the origin. Let
$\pi : X \to B$ be an embedded resolution of $(\f,0)$. We denote
by $E_i, i\in J$, the irreducible components of
$\pi^{-1}(\f)_{\text{red}}$. 
For every  $i\in J$, let $N_i$ and $\nu_i-1$ be the multiplicities  of  $E_i$ 
in the divisor of respectively  
 $f \circ \pi$ and $\pi^*(dx\wedge dy)$ on $X$. 
One has that $N_i$ and $\nu_i$ belong to $\mathbb{N}^*$ and if $E_i$ is an irreducible component 
of the strict transform of $\f$ then  $\nu_i=1$. Denote also 
$\mathring{E}_i:= E_i \setminus
\left(\cup_{j\ne i} E_j \right)$ for $i \in J$. Then one has the following interpretation of the $R(n,b_1,\ldots,b_g;t)$ 
\begin{equation*}
R(n,b_1,\ldots,b_g;t)=
t- \sum_{i\in J,\, E_i \ne \tilde{C}} \chi(\mathring{E}_i)     t^{\nu_i/N_i}\frac{1-t}{1-t^{1/N_i}}
\end{equation*}
where  $\tilde C$ is the unique strict transform of $\f$.
For a  vertex $i$ of the minimal embedded resolution graph
 its valency $\delta_i$ is the number of 
adjacent vertices to it.  A vertex is called 
a \emph{rupture vertex} if its valency  is  at least $3$.
Most of the vertices in the resolution graph have valency $2$ and since  the corresponding exceptional divisors $E_i$
 are rational curves $\chi(\mathring{E}_i)=0$. Furthermore in this case the valency of the vertex are either $1$, $2$ or $3$.

The shape of the minimal embedded resolution graph in this case
is the same  as the  Eisenbud-Neumann splice diagram
(cf. \cite[page 49]{EN}). 
If the germ $(C,0)$ has $g$  Newton pairs
$\{(p_k,q_k)\}_{k=1}^g$
with gcd$(p_k,q_k)=1$ and $p_k\geq 2$ and $q_k\geq 1$
(and by convention, $q_1>p_1$),
define the integers $\{a_k\}_{k=1}^g$ by
$a_1:=q_1$ and $a_{k+1}:=q_{k+1}+p_{k+1}p_ka_k$ for  $k\geq 1$.
Then its  Eisenbud-Neumann splice diagram decorated by
the following splice data $\{(p_k,a_k)\}_{k=1}^g$  and has the following  shape:

\begin{figure}[ht]
\begin{center}
\iftikz
\begin{tikzpicture}[scale=.85,vertice/.style={draw,circle,fill,minimum size=0.2cm,inner sep=0}]
\coordinate (M1) at (0,0);
\coordinate (M2) at (2,0);
\coordinate (Z) at (0,1.5);
\coordinate (M3) at ($(M2)-(Z)$);
\coordinate (M4) at ($2*(M2)$);
\coordinate (M5) at ($(M4)-(Z)$);
\coordinate (M6) at ($3*(M2)$);
\coordinate (M7) at ($4*(M2)$);
\coordinate (M8) at ($5*(M2)$);
\coordinate (M9) at ($(M8)-(Z)$);
\coordinate (M10) at ($6*(M2)$);
\coordinate (M11) at ($(M10)-(Z)$);
\coordinate (M12) at ($7*(M2)$);

\draw (M1)--(M2) node[above left] {$a_1$};
\draw  (M2)  --(M3) node[pos=.3,right] {$p_1$} ;
\draw (M2)--(M4) node[above left] {$a_2$};
\draw  (M4)  --(M5) node[pos=.3,right] {$p_2$} ;
\draw (M4)--(M6);
\draw[dotted] (M6)--(M7);
\draw (M7)--(M8)node[above left] {$a_{g-1}$};
\draw  (M8)  --(M9) node[pos=.3,right] {$p_{g-1}$} ;
\draw (M8)--(M10)node[above left] {$a_{g}$};
\draw  (M10)  --(M11) node[pos=.3,right] {$p_{g}$} ;
\draw[-{[scale=2]>}] (M10)--(M12);

\node[vertice] at (M1) {};
\node[vertice] at (M2) {};
\node[vertice] at (M3) {};
\node[vertice] at (M4) {};
\node[vertice] at (M5) {};
\node[vertice] at (M8) {};
\node[vertice] at (M9) {};
\node[vertice] at (M10) {};
\node[vertice] at (M11) {};
\end{tikzpicture}
\else
\includegraphics{en}
\fi
\caption{}
\end{center}
\end{figure}

The $g$ rupture components $\tilde{E}_1,\ldots, \tilde{E}_g$, ordered from the left to the right of the resolution graph are the same as in the splice diagram and their
 numerical data can be computed inductively  from the 
\[\begin{array}{ll}
{\tilde{N}_k}:=a_k\cdot p_k\cdot p_{k+1}\cdot\ldots\cdot p_g & \mbox{for $1\leq k\leq g$};\\
\tilde{\nu}_k:=p_k\tilde{\nu}_{k-1} +q_k      & \mbox{where $\tilde{\nu}_0=1$}, \end{array}
\]

The numerical data associated to the components $g+1$ components of valency 1
$E_0,E_1, \ldots, E_g$, here $E_0$ is the most left hand side vertex corresponding to the first blow-up
and its numerical data is equal to $(N_0,\nu_0)=(n,2)$ with  $n=p_1p_2\cdots p_g$.
The numerical data associated to other valency one components can be also computed from 
\[\begin{array}{ll}
{N_k}=a_k\cdot p_{k+1}\cdot\ldots\cdot p_g   
& \mbox{for $1\leq k\leq g$};\\
{\nu}_k=\tilde{\nu}_{k-1} + \lceil \frac{q_k}{p_k} \rceil  & \mbox{for $1\leq k\leq g$}
\end{array}
\]

\subsection{Yano's conjecture for isolated germs of plane curves}
\mbox{}

A natural extension of the Yano conjecture for isolated plane curve singularity germ could be 
the following conjecture

\begin{yanoextcj}
For almost all isolated plane curve singularity germ $f:(\bc^2,0)\to (\bc,0)$ 
  with isolated singularity and Milnor number $\mu$, the $b$-exponents  
$\{\tilde{\beta}_1\!,\ldots,\!\tilde{\beta}_\mu\}$ 
 are given by  the generating series
\begin{equation}\label{conjecture1}
 \sum_{i=1}^{\mu} t^{\tilde{\beta}_i}=
t+\sum_{i}(\delta_i -2) \left( t^{\nu_i/N_i}\frac{1-t}{1-t^{1/N_i}} 
 \right),
\end{equation}
showing how $b$-exponents depends on the topology of $f$. 
\end{yanoextcj}

\begin{ejm}\label{ejm:val4}
Let 
$f(x,y)=y^4-x^6$
be  a germ  with two $\mathbb{A}_2$-singularities having
 intersection number equals 6.
The minimal embeded resolution   graph has 3 exceptional divisors $E_1, E_2,E_3$ with numerical data $(N, \nu,\delta)$ given respectively by equals $(4,2,1)$, $(6,3,1)$ and 
$(12,5,4)$.
Then \eqref{conjecture1}
 equals
$$
t+2 \left( t^{5/12}\frac{(1-t)}{(1-t^{1/12})} 
 \right)- \left( t^{2/4}\frac{(1-t)}{(1-t^{1/4})} +  t^{3/6}\frac{(1-t)}{(1-t^{1/6})}\right) 
$$
equals 
$$
t+t^{4/3}+t^{5/4}+t^{7/6}+2 t^{13/12}+2t^{11/12}+t^{5/6}+t^{3/4}+t^{2/3}+2t^{7/12}+
2t^{5/12}.
$$
Using \texttt{Singular}~\cite{DGPS} inside~\cite{sage}, a  $\mu$-constant  versal deformation of $f$ is given by 
 $g(x,y,a,b):=f + a x^3 y^2 +b x^4  y^2 $ and the  Bernstein-Sato  polynomial of $g$ for random values  of $a$ and $b$ is equal to
$$
-17/12,\!-4/3,\!-5/4,-7/6,-13/12,-1,-11/12,-5/6,-3/4,-2/3,-7/12,-5/12,
$$ 
so that they do not coincide. 

This can be 
confirmed using \texttt{checkRoot} for $s=-17/12$  of~\cite{LMM:12} 
in \texttt{Singular}~\cite{DGPS}, where the base field is $\mathbb{C}(a,b)$.
Moreover, it can be proved that for general~$a, b$ the Tjurina number equals
the expected value for Hertling-Stahlke bound, i.e., $14$; using~\cite{LP}
the values of Tjurina number are constant in these $\mu$-constant strata.
\end{ejm}

The previous example shows that the proposed conjecture may not hold 
when there are vertices with valency greater than~$3$. 
Based on the irreducible case we want to study the conjecture for
the case where valencies are at most~$3$.

\begin{yanoextcj1}
Let $\Sigma_\mu$ be the $\mu$-constant stratum of a germ $f:(\bc^2,0)\to (\bc,0)$ 
of isolated singularity, such that no eigenvalue $\zeta\neq 1$ of the monodromy is
mutiple (in particular the valency of the vertices of the resolution graph
is at most~$3$). Then the $\mu$  $b$-exponents  
$\{\tilde{\beta}_1\!,\ldots,\!\tilde{\beta}_\mu\}$ of a generic element of~$\Sigma_\mu$ 
 are given by  the generating series \eqref{conjecture1}
\end{yanoextcj1}

Most probably, the hypothesis on the monodromy can be replaced \emph{no repeated non-integral exponent of the monodromy}
as the result in~\cite{Pi88} for non-degenerate Newton polynomial germs suggests; some examples in the last section
go in the same direction. The condition on the valency seems to be more essential, due to Example~\ref{ejm:val4}.

\subsection{Singularities with non-degenerated principal part and commode}
\mbox{}

Assume that the power series $f$ has non-degenerated principal part and denote its Newton polygon at $0$ by
 $\Gamma_f$ ,  
with $\ell$ facets and commode ($\Gamma_f$ meets with $x=0$ at 
$(0,\tau_0 )$ and with $y=0$ at 
$(\sigma_0,0)$).  
We also assume that the set $\spec(f)$ of  spectral numbers are distinct.

Assume that  $f_i(x,y)=1$, with $f_i(x,y)=\frac{c_i x+d_i y}{n_i},$
is the equation of the facet $F_i$ of  $\Gamma_f$ so that  
$\gcd(c_i, d_i,n_i)=1$, $1\leq i\leq \ell$.

Set
$$\mathcal{N}=\{q\in \bq:\, \sigma_0 q\in \bn\, or\, \tau_0 q \in \bn\,\}.
$$
Let $b_f$ be the monic polynomial such that its roots are the rational numbers
$\sigma_{i,k}:=-\frac{c_i+d_i+k}{n_i}:$ with $0\leq k< n_i$ and for all facet $F_i$ such that $\sigma_{i,k}\notin \mathcal{N}.$

\begin{thm}[{\cite[Theorem~1]{Pi88}}]\label{nondeg}
 For almost all germs of plane curves  which have $\Gamma_f$ as
Newton polygon at the origin and all non-integral elements in $\spec(f)$ are distinct  then $f$ admits $b_f$ as Bernstein-Sato polynomial.   
\end{thm}

Note that Example~\ref{ejm:val4} does not satisfy the hypotheses of the above theorem.
The minimal embeded resolution   graph of  germs  in Theorem~\ref{nondeg}  has all 
exceptional divisors  of valencies exactly $1,2$ and $3$.
There are at most $2$ divisors with valency $1$ and $\ell$ divisors of valency $3$.  
For all  $1\leq i\leq \ell$, let $E_i$ be the corresponding  divisor has  numerical data $(N_i, \nu_i,\delta_i)=(n_i,c_i+d_i,3)$. 
So that the roots in this case appear as  in the EN-diadram of the germ.  
So that a generic equisingular deformation of $f$ admits $b_f$ as Bernstein-Sato polynomial.  

If two spectral numbers are congruent $\bmod\, \bz$, their difference is $\pm 1$, and they correspond
to a $2$-Jordan block of the monodromy, so we can recover the $b$-exponents from
the Bernstein-Sato polynomial.

\begin{cor}
If the germ $f$ is Newton non-degenerated with respect to its Newton polygon, commode and all 
the spectral numbers are distinct then for a generic equisingular deformation of $f$
the $b$-exponents are given by~\eqref{conjecture1}.
\end{cor}

\section{Improper integrals}

Most of the results in this section come from~\cite{ACLM-Yano2}. We start with
$1$-variable improper integrals.

\begin{prop}\label{prop:int_una_var1} 
Let $f:[0,1]\times\bc\to\bc$ be
an analytic function. Then the function
\[
s\mapsto\int_0^1 f(t,s) t^s \frac{dt}{t}
\]
is holomorphic on $\Re s>0$ and admits a meromorphic continuation to~$\bc$ with poles
contained in $\bz_{\leq 0}$. Moreover, if $f(t,s)$ is algebraic whenever $t$ is algebraic and $s$ rational,
then, the residues are algebraic.
\end{prop}

If the function~$f$ is independent of $s$, then the above function will be denoted by $G_f(s)$.
Let us consider now the $2$-variable case.

\begin{prop}\label{Essou1a} Let $f\in\br[x,y]$ such that $f>0$ in~$[0,1]^2$ and let
$a_1,b_1,a_2,b_2\in\mathbb{Z}_{\geq 0}$ (by convention $\frac{b_i}{a_i}=+\infty$ if $a_i=0$). The function 
\[
s\mapsto\int_0^1  \int_0^1  f(x,y)^s x^{a_1s+b_1} y^{a_2s+b_2} \frac{dx}{x} \frac{dy}{y}.
\]
is holomorphic in $\Re s>\max\left(-\frac{b_1}{a_1},-\frac{b_2}{a_2}\right)$ and admits
 a meromorphic continuation on $\bc$, where the set of poles is a subset of 
$S=\left\{ -\frac{b_1+\nu_1}{a_1}, \, \, \nu_1 \in \mathbb{Z}_{\geq 0} \,\right\} \cup 
\left\{ -\frac{b_2+\nu_2}{a_2}, \, \, \nu_2\in \mathbb{Z}_{\geq 0} \,\right\}$.
\end{prop}

We can be more explicit on those poles.

\begin{prop}\label{continuation} 
With the hypotheses of Proposition{\rm~\ref{Essou1a}}, let $\alpha\in S$.

\begin{enumerate}[label=\rm(P\arabic*)]
\item If $\alpha=-\frac{b_1+\nu_1}{a_1}$ for some $\nu_1\in\bz_{\geq 0}$ and $\alpha\neq -\frac{b_2+\nu_2}{a_2}$
$\forall\nu_2\in\bz_{\geq 0}$, then the pole is of order at most one and its residue equals
\begin{equation*}
\frac{1}{\nu_1! a_1} 
G_{ h_{\nu_1,\alpha,x}}(a_2\alpha+b_2),\quad
h_{\nu_1,\alpha,x}(y):=\frac{\partial^{\nu_1} f^\alpha}{\partial x^{\nu_1}}(0,y).
\end{equation*}
\item If $\alpha=-\frac{b_2+\nu_2}{a_2}$ for some $\nu_2\in\bz_{\geq 0}$ and $\alpha\neq -\frac{b_1+\nu_1}{a_1}$
$\forall\nu_1\in\bz_{\geq 0}$, then the pole is of order at most one and its residue equals
\begin{equation*}
\frac{1}{\nu_2! a_2} 
G_{ h_{\nu_2,\alpha,y}}(a_1\alpha+b_1),\quad
h_{\nu_2,\alpha,y}(x):=\frac{\partial^{\nu_2} f^\alpha}{\partial y^{\nu_2}}(x,0).
\end{equation*}
\item If $\alpha=-\frac{b_1+\nu_1}{a_1}=-\frac{b_2+\nu_2}{a_2}$ for some $\nu_1,\nu_2\in\mathbb{Z}_\geq 0$,
then the pole is of order at most~$2$ and the coefficient of $(s-\alpha)^{-2}$ in the Laurent expansion is
$$
\frac{1}{\nu_1! \nu_2!a_1a_2}
\frac{\partial^{\nu_1+\nu_2} f^{\alpha}}{\partial x^{\nu_1}\partial y^{\nu_2}}(0,0).
$$
\item If in the previous situation the pole is of order at most one, then
the continuation of the functions $G_{ h_{\nu_1,\alpha,x}}$ and $G_{ h_{\nu_2,\alpha,y}}$
are holomorphic at $a_2\alpha+b_2$ and $a_1\alpha+b_1$, respectively and its residue equals
\begin{gather*}
\frac{1}{\nu_1! a_1} 
G_{ h_{\nu_1,\alpha,x}}(a_2\alpha+b_2)+\frac{1}{\nu_2! a_2} 
G_{ h_{\nu_2,\alpha,y}}(a_1\alpha+b_1).
\end{gather*}
\end{enumerate}
\end{prop}

The last result does not appear in~\cite{ACLM-Yano2} but it can be deduced easily.
The following lemma is useful for the residue computations.

\begin{lema}\label{beta}
Let $p\in \mathbb{N}$ and $c\in \mathbb{R}_{>0}$. 
Given $s_1, s_2\in \bc$ 
  such that  $-\alpha=s_1+s_2>0$
then 
\begin{equation}
G_{ \left( y^{p}+c\right)^\alpha}(p s_1)+
G_{ \left(1 +c x^{p}\right)^\alpha}(ps_2)=
\frac{c^{-s_2}}{p}\boldsymbol{B}\left( s_1, s_2\right)
\end{equation}
where $\boldsymbol{B}$ is the \emph{beta function}.
\end{lema}

In \cite{ACLM-Yano2}, we proceeded as follows. For a fixed equisingularity type, we consider 
\emph{generic} polynomial representatives~$f$ with real algebraic coefficients, in some field~$\mathbb{K}$, and such that
for a suitable semi-algebraic compact domain~$\mathcal{D}$, we had $f>0$ in $\mathcal{D}\setminus\{(0,0)\}$
(the origin is in the boundary of~$\mathcal{D}$). For a special choice of coordinates and a \emph{weight}
 function~$g$ we consider the following integrals
\begin{equation}\label{eq:Ifg} 
\mathcal{I}(f,g,\beta_1,\beta_2,\beta_3)(s):=\int_{\mathcal{D}} f(x,y)^s x^{\beta_1} y^{\beta_2} g(x,y)^{\beta_3}
\frac{dx}{x}\frac{dy}{y}
\end{equation}
where $\beta_1,\beta_2,\beta_3+1\in\bz_{>0}$. 
These integrals are holomorphic in a semiplane of~$\bc$ and admitted a meromorphic continuation 
(see Example~\ref{ejm:34-35} for an idea of the proof). The knowledge of the residues allowed
us to prove the following theorem.

\begin{thm}\label{pole-integral-root}
Let $f\in \mathbb{K}[x,y]$ be as above. 
Let $\alpha$ be a pole of  $\mathcal{I}_{}(f,\beta_1,\beta_2,\beta_3)(s)$  with transcendental residue,  
and such that $\alpha+1$ is not a pole of $\mathcal{I}_{}(f,\beta'_1,\beta'_2,\beta'_3)(s)$ 
for any $(\beta'_1,\beta'_2,\beta'_3)$. Then $\alpha$ is a root of
the Bernstein-Sato polynomial  $b_{f}(s)$ of $f$.
\end{thm}

\section{Partial proof of the conjecture}

We are going to prove the modified extended conjecture when the number of rupture vertices
is small.

\begin{thm}
The extended Yano's conjecture holds for germs of plane
curve singularities with no multiple eigenvalues
of the monodromy (except maybe~$1$), and such that
there are at most two rupture vertices and their valency
is at most~$3$.
\end{thm}

\begin{proof}[Sketch of the proof]
As we have seen in Example~\ref{ejm:val4}, the valency condition and the 
non-existence of multiple values distinct from~$1$ seem to be essential.
The condition of $1$ or $2$ branching vertices is only technical.

There are three types of such singularities.

\begin{enumerate}[label=(S\arabic{enumi})]
\item\label{S1} The resolution graph is linear.
\item\label{S2} The germ is the product of two irreducible germs with one-Puiseux pair $(m,n)$ and 
intersection number $>mn$, and eventually two smooth branches with intersection numbers $m,n$ with the singular
branches.
\item\label{S3} The resolution graph coincides with the one of a two-Puiseux pair irreducible (which is part of the germ).
\end{enumerate}

The case \ref{S1} is a consequence of \cite[Theorem~1]{Pi88}. 
The case \ref{S2} is represented by the $\mu$-constant versal deformation of
$f=x^\epsilon y^\eta ((y^m-x^n)^2-x^u y^v)$, where $\epsilon,\eta\in\{0,1\}$ and $u,v$ depend on the intersection
number of the two singular branches. We omit the cases where there are multiple eigenvalues distinct from~$1$.
We follow the strategy in \cite{ACLM-Yano2}. The presence of $x,y$ does not affect  this strategy
as we explain later for~\ref{S3}. 
If there are more than $2$ branches, $1$ is a multiple eigenvalue of the monodromy. Nevertheless,
the only point where this condition is needed is for Varchenko's lower semicontinuity~\cite{V80} and only eigenvalues
distinct from~$1$ cannot be multiple for this result.

Let us finish with \ref{S3}.
Let us consider the improper integral $\mathcal{I}(f,g,\beta_1,\beta_2,\beta_3)$ of~\eqref{eq:Ifg},
studied in \cite{ACLM-Yano2},
where $\beta_1,\beta_2,\beta_3+1\in\mathbb{Z}_{>0}$, $f,g$ are real polynomials positive on $[0,1]^2\setminus\{(0,0)\}$, $f$ is a $2$-Puiseux-pair germ singularity for 
which the Newton polygone is of type $(y^m\pm x^n)^p$, $g$ is a $1$-Puiseux pair singularity with Newton polygone
$y^m\pm x^n$ and maximal contact with~$f$. For~\ref{S3} we replace~$f$ by $x^\epsilon y^\eta f g^{\gamma}$, 
$\epsilon,\eta,\gamma\in\{0,1\}$. We repeat the process as in~\cite{ACLM-Yano2}.
\end{proof}

\section{Computations on examples with multiple eigenvalues}

\begin{ejm}
Let us consider $f(x,y)=y^5+x^2 y^2+x^5$; its $\mu$-constant miniversal deformation is a singleton, so its Bernstein-Sato
polynomial coincides with the generic one. This singularity does not satisfy
\cite[Theorem 1]{Pi88} since the exponents  $\pm\frac{1}{10},\pm\frac{3}{10}$ appear twice ($\pm\frac{1}{2}$ appear only once). Using \texttt{Singular}, the Bernstein polynomial is
\[
\left(s + \frac{1}{2}\right)^2 \left(s + \frac{7}{10}\right) \left(s + \frac{9}{10}\right) \left(s + 1\right) \left(s + \frac{11}{10}\right) \left(s + \frac{13}{10}\right).
\]
\begin{figure}[ht]
\iftikz
\begin{tikzpicture}[scale=.9,vertice/.style={draw,circle,fill,minimum size=0.2cm,inner sep=0}]
\coordinate (M1) at (0,0);
\coordinate (M2) at (2,0);
\coordinate (Z) at (0,1.5);
\coordinate (M3) at ($(M2)+(Z)$);
\coordinate (M4) at ($2*(M2)$);
\coordinate (M6) at ($3*(M2)$);
\coordinate (M7) at ($3*(M2)+(Z)$);
\coordinate (M8) at ($4*(M2)$);

\draw node[below] {$(5,3)$} (M1)--(M2);
\draw[-{[scale=2]>}]  (M2) node[below] {$(10,5)$}  -- (M3) ;
\draw (M2)--(M4) node[below] {$(4,2)$};
\draw (M4)--(M6) node[below] {$(10,5)$};
\draw[-{[scale=2]>}] (M6)--(M7);
\draw (M6)--(M8)node[below] {$(5,3)$};

\node[vertice] at (M1) {};
\node[vertice] at (M2) {};
\node[vertice] at (M4) {};
\node[vertice] at (M6) {};
\node[vertice] at (M8) {};
\end{tikzpicture}
\else
\includegraphics[scale=1]{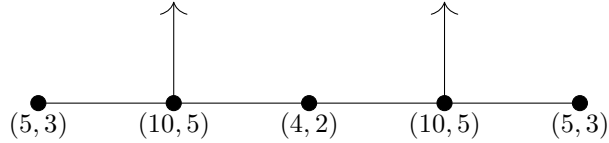}
\fi
\caption{Resolution graph of $y^5+x^2 y^2 +x^5$ with $(N,\nu)$-data.}
\label{fig:acampo23-23}
\end{figure}
The extended  conjecture is satisfied even though we are not in the hypotheses of the modified one.
\end{ejm}

\begin{ejm}
Let us consider $f(x,y)=y^5+x^2 y^2+x^7$; its $\mu$-constant versal deformation is also a singleton, so its Bernstein
polynomial coincides with the generic one. This singularity \emph{does} satisfy
\cite[Theorem 1]{Pi88} since $\pm\frac{1}{2}$ appear as exponents of the monodromy, even though 
$\exp\left(2i\pi\frac{\pm1}{2}\right)=-1$ is a double eigenvalue.  
Using \texttt{Singular}, 
we can confirm the expected Bernstein-Sato polynomial.
\end{ejm}

\begin{ejm}\label{ejm:34-35}
Let us consider $f(x,y)=x^{3} y^{3} + x^{7} + y^{8}$;
a $\mu$-constant versal deformation is given by $f_{t,s}(x,y):=x^{3} y^{3} + x^{7} + t x^{6} y + s x y^{7} + y^{8}$.
As in the previous example the hypotheses  of \cite[Theorem~1]{Pi88} are satisfied and hence the
extended conjecture holds; note that there are multiple eigenvalues for the monodromy but
the exponents of the monodromy are distinct. 

\begin{figure}[ht]
\begin{center}
\iftikz
\begin{tikzpicture}[scale=.9,vertice/.style={draw,circle,fill,minimum size=0.2cm,inner sep=0}]
\coordinate (M1) at (0,0);
\coordinate (M2) at (2,0);
\coordinate (M0) at ($-1*(M2)$);
\coordinate (Z) at (0,1.5);
\coordinate (M3) at ($(M2)+(Z)$);
\coordinate (M4) at ($2*(M2)$);
\coordinate (M6) at ($3*(M2)$);
\coordinate (M7) at ($4*(M2)+(Z)$);
\coordinate (M8) at ($4*(M2)$);
\coordinate (M9) at ($5*(M2)$);

\draw (M0)node[below] {$(7,3)$} --(M1) ;
\draw node[below] {$(14,6)$} (M1)--(M2);
\draw[-{[scale=2]>}]  (M2) node[below] {$(21,7)$}  -- (M3) ;
\draw (M2)--(M4) node[below] {$(6,2)$};
\draw (M4)--(M6) node[below] {$(15,5)$};
\draw[-{[scale=2]>}] (M8)--(M7);
\draw (M6)--(M8)node[below] {$(24,8)$};
\draw (M6)--(M9)node[below] {$(8,3)$};

\node[vertice] at (M0) {};
\node[vertice] at (M1) {};
\node[vertice] at (M2) {};
\node[vertice] at (M4) {};
\node[vertice] at (M6) {};
\node[vertice] at (M8) {};
\node[vertice] at (M9) {};
\end{tikzpicture}
\else
\includegraphics[scale=1]{acampo3435}
\fi
\end{center}
\end{figure}

Yano's candidates start at $\frac{1}{3}=\frac{7}{21}$ and $\frac{1}{3}=\frac{8}{24}$.
The particular Bernstein-Sato polynomials
may depend on $s,t$; let us study some jumps using improper integrals.
Choose $t,s\in\mathbb{R}_{\geq 0}$;
note that $f_{t,s}>0$ in $[0,1]^2\setminus\{(0,0)\}$.
Let us denote, for $\beta_1,\beta_2\in\bz_{\geq 1}$:
\[
\mathcal{I}_{\beta_1,\beta_2}=\int_{[0,1]^2}
f_{t,s}(x,y)^s x^{\beta_1} y^{\beta_2}\frac{dx}{x}\frac{dy}{y}
\]
Let us decompose this square in two domains:
\[
\{(x,y)\in[0,1]^2\mid x^{\frac{4}{3}}\leq y\leq 1\},\quad
\{(x,y)\in[0,1]^2\mid 0\leq y\leq x^{\frac{4}{3}}\}.
\]
Integrating on each subdomain we decompose $\mathcal{I}_{\beta_1,\beta_2}=
\mathcal{I}_{1,\beta_1,\beta_2}+\mathcal{I}_{2,\beta_1,\beta_2}$.

With suitable change of variables, we have:
Let us consider the change of variables $x\mapsto x y^3$, $y\mapsto y^4$:
\begin{gather*}
x\mapsto x y^3,\quad y\mapsto y^4\Longrightarrow
\mathcal{I}_{1,\beta_1,\beta_2}=4\int_{[0,1]^2}
\tilde{f}_{t,s}(x,y)^s x^{\beta_1} y^{3\beta_1+4\beta_2+21s} \frac{dx}{x}\frac{dy}{y}
\end{gather*}
where
\begin{gather*} 
\tilde{f}_{t,s}(x,y):=t x^{6} y + s x y^{10} + x^{7} + x^{3} + y^{11}.
\end{gather*}
In the same way
is $x\mapsto x^{3}$, $y\mapsto x^4 y$;
\[
x\mapsto x^{3},\quad y\mapsto x^4 y\Longrightarrow
\mathcal{I}_{2,\beta_1,\beta_2}=3\int_{\mathcal{D}_2}
{f}^*_{t,s}(x,y)^s x^{3\beta_1+4\beta_2+21s} y^{\beta_2} \frac{dx}{x}\frac{dy}{y}.
\]
where
\begin{gather*} 
{f}^*_{t,s}(x,y):=t x y + s x^{10} y^{7} + x^{11} y^{8} + y^{3} + 1.
\end{gather*}
Note that $I_{2,\beta_1,\beta_2}$ satisfies the hypotheses of Proposition~\ref{Essou1a}, which was the goal
of these changes of variables. Since it is not the case for $I_{1,\beta_1,\beta_2}$, let 
us perform a decomposition of the square as
\[
\{(x,y)\in[0,1]^2\mid 0\leq y\leq x^{\frac{3}{11}}\},\quad
\{(x,y)\in[0,1]^2\mid x^{\frac{3}{11}}\leq y\leq 1\},
\]
and denote the corresponding integral decomposition as $I_{1,\beta_1,\beta_2}=I_{1,1,\beta_1,\beta_2}+
I_{1,2,\beta_1,\beta_2}$. Suitable changes of variables yield:
\[
x\mapsto x^{11},\ y\mapsto x^3 y\Longrightarrow
\mathcal{I}_{1,1,\beta_1,\beta_2}=\!44\!\!\int_{\mathcal{D}}
\!\hat{f}_{t,s}(x,y)^s x^{4(5\beta_1+3\beta_2+24s)} y^{3\beta_1+4\beta_2+21s} \frac{dx}{x}\!\frac{dy}{y},
\]
where
\begin{gather*} 
\hat{f}_{t,s}(x,y):=t x^{36} y + s x^{8} y^{10} + x^{44} + y^{11} + 1,
\end{gather*}
and 
\[
x\mapsto x y^{11},\quad y\mapsto y^3\Longrightarrow
\mathcal{I}_{1,2,\beta_1,\beta_2}=12\int_{\mathcal{D}}
\check{f}_{t,s}(x,y)^s x^{\beta_1} y^{{4(5\beta_1+3\beta_2+24s)}} \frac{dx}{x}\frac{dy}{y},
\]
where
\begin{gather*} 
\check{f}_{t,s}(x,y):=t x^{6} y^{36} + s x y^{8} + x^{7} y^{44} + x^{3} + 1.
\end{gather*}
The candidate pole $-\frac{8}{21}$ can be pole only for $\beta_1=\beta_2=1$, and in this
case the residue is
\begin{gather*}
\frac{44}{21}\int_0^1\frac{\partial\hat{f}^{-\frac{8}{21}}}{\partial y}(x,0) x^{-\frac{32}{7}}\frac{dx}{x}
+
\frac{3}{21}\int_0^1\frac{\partial\check{f}^{-\frac{8}{21}}}{\partial x}(0,y) y^{}\frac{dy}{y}=\\
-\frac{8\cdot 44 t}{21^2}\int_0^1(1+x^{44})^{-\frac{29}{21}} x^{\frac{220}{7}}\frac{dx}{x}
-\frac{3\cdot 8 t}{21^2}\int_0^1(1+y^{3})^{-\frac{29}{21}} y^{2}\frac{dy}{y}=\\
-\frac{8 t}{21^2}\int_0^1(1+u)^{-\frac{29}{21}} u^{\frac{5}{7}}\frac{du}{u}
-\frac{ 8 t}{21^2}\int_0^1(1+u)^{-\frac{29}{21}} u^{\frac{2}{3}}\frac{du}{u}=
-\frac{ 8 t}{21^2}\boldsymbol{B}\left(\frac{5}{7},\frac{2}{3}\right).
\end{gather*}
Hence, for $t\neq 0$ (and algebraic), we have that $-\frac{8}{21}$ is a root of the Bernstein-Sato polynomial. 
Note that we can prove that $-\frac{29}{21}$ is a pole of $\mathcal{I}_{7,2}$ with transcendental
residue for any (algebraic) value of $t,s$. In particular, $-\frac{29}{21}$ is a root of the
Bernstein polynomial if $t=0$ and $s$ is algebraic after Theorem~\ref{pole-integral-root}.
Note that  $-\frac{8}{21}$  and $-\frac{29}{21}$ cannot be simultaneously roots
of the Bernstein-Sato polynomial, since $\exp\left(-2i\pi\frac{8}{21}\right)=\exp\left(-2i\pi\frac{29}{21}\right)$
is a simple eigenvalue of the monodromy. 
These results are confirmed by \texttt{Singular} and \texttt{checkRoot}. We have then proved that there is a function $f_0$ in 
the $\mu$-constant stratum such that $-\frac{8}{21}$ is not a root of Bernstein-Sato polynomial for  $f_0$, compare with~\cite{ACLM-Yano2a}
\end{ejm}

\begin{ejm}
Let us consider $f_\pm(x,y):=(x^4-y^3)^2+x^{6} y^2$ which corresponds to the case \ref{S3}.
A $\mu$-constant versal deformation is given by $f_{\mathbf{t}}(x,y)=f_{\pm}(x,y)+t_1 x^8 y+t_2 x^9$.
Let $\mathcal{D}:=\{(x,y)\in[0,1]^2\mid 0\leq y\leq x^{\frac{4}{3}}\}$ and for $t,s\in\br_{\geq 0}$, consider
\[
\mathcal{I}_{\beta_1,\beta_2,\beta_3}:=
\int_{\mathcal{D}} f_{\mathbf{t}}(x,y)^s x^{\beta_1} y^{\beta_2} (x^4-y^3)^{\beta_3} \frac{dx}{x}\frac{dy}{y}
\]
for $\beta_1,\beta_2,\beta_3+1\in\bz_{>0}$. In order to check that it is holomorphic with
meromorphic continuation, we perform a first change of variable:
\[
x\mapsto\! x^3, y\mapsto x^4(1-y)\Longrightarrow
\mathcal{I}_{\beta_1,\beta_2,\beta_3}\!\!=\!\!3\!\!\!
\int_{[0,1]^2}\!\!\! \tilde{f}_{\mathbf{t}}(x,y)^s x^{3\beta_1+4\beta_2+12\beta_3+24s} y^{\beta_3+1} q(y) \frac{dx}{x}\!\frac{dy}{y}
\]
where $q(y):=(1-y)^{\beta_2-1} (3-3y+y^2)^{\beta_3}$ and 
\[
\tilde{f}_{\mathbf{t}}(x,y)=y^2(3-3 y+y^2)^2+x^{2} (1-y)^2+ t_1 x^{4} (1-y)+t_2 x^{3}.
\]
\begin{figure}[ht]
\begin{center}
\iftikz
\begin{tikzpicture}[scale=.9,vertice/.style={draw,circle,fill,minimum size=0.2cm,inner sep=0}]
\coordinate (M1) at (0,0);
\coordinate (M2) at (2,0);
\coordinate (M0) at ($-1*(M2)$);
\coordinate (Z) at (0,1.5);
\coordinate (M3) at ($(M2)+(Z)$);
\coordinate (M4) at ($2*(M2)$);
\coordinate (M6) at ($0.5*(M2)+2*(Z)$);
\coordinate (M7) at ($1.5*(M2)+2*(Z)$);
\coordinate (M8) at ($4*(M2)$);
\coordinate (M9) at ($5*(M2)$);

\draw (M0)node[below] {$(8,3)$} --(M1) ;
\draw node[below] {$(16,5)$} (M1)--(M2);
\draw  (M2) node[below] {$(24,7)$}  -- (M3) node[right] {$(26,8)$};
\draw (M2)--(M4) node[below] {$(6,2)$};
\draw[-{[scale=2]>}] (M3)--(M6);
\draw[-{[scale=2]>}] (M3)--(M7);

\node[vertice] at (M0) {};
\node[vertice] at (M1) {};
\node[vertice] at (M2) {};
\node[vertice] at (M3) {};
\node[vertice] at (M4) {};
\end{tikzpicture}
\else
\includegraphics[scale=1]{tangente-23-1}
\fi
\caption{}
\label{fig:tangente-23-1}
\end{center}
\end{figure}
Decomposing the square in two triangles with the diagonal line, we can decompose 
$\mathcal{I}_{\beta_1,\beta_2,\beta_3}=\mathcal{I}_{1,\beta_1,\beta_2,\beta_3}+\mathcal{I}_{2,\beta_1,\beta_2,\beta_3}$;
with the following changes of variables we obtain
\[
x\mapsto x, y\mapsto x y\Longrightarrow
\mathcal{I}_{1,\beta_1,\beta_2,\beta_3}\!\!=\!\!3\!\!
\int_{[0,1]^2}\! \hat{f}_{\mathbf{t}}(x,y)^s x^{3\beta_1+4\beta_2+13\beta_3+1+26s} y^{\beta_3+1} q(xy) \frac{dx}{x}\frac{dy}{y}
\]
and $x\mapsto x y,\ y\mapsto y\Longrightarrow$:
\[
\mathcal{I}_{2,\beta_1,\beta_2,\beta_3}\!\!=\!\!3\!\!
\int_{[0,1]^2}\! \check{f}_{\mathbf{t}}(x,y)^s x^{3\beta_1+4\beta_2+12\beta_3+24s} y^{3\beta_1+4\beta_2+13\beta_3+1+26s} q(y) \frac{dx}{x}\frac{dy}{y},
\]
where
\begin{gather*} 
\hat{f}_{\mathbf{t}}(x,y)=y^2(3-3 xy+x^2y^2)^2+ (1-xy)^2+ t_1 x^2 (1-xy)+t_2 x,\\
\tilde{f}_{\mathbf{t}}(x,y)=(3-3 y+y^2)^2+x^{2} (1-y)^2+ t_1 x^{4} y^2  (1-y)+t_2 x^{3} y.
\end{gather*}
\end{ejm}

\begin{ejm}
A $\mu$-constant miniversal deformation for $f(x,y)=(y^2-x^3)^2+x^{12}$ is constant. It does not satisfy the hypotheses
of the modified extended conjecture, since there are multiple eigenvalues (and multiple exponents of the monodromy)
but, nevertheless, the extended conjecture holds.
\end{ejm}

\begin{ejm}\label{ejm:acampo-23-210}
Let $f(x,y):=x(y^3-x^2)(y^{2}-x^{10})$, with $\mu$-constant miniversal deformation $f_t(x,y):=f(x,y)+t y^{7}$.
This example has multiple eigenvalues (besides~$1$) and it is a counterexample for the extended conjecture.
It is not hard to prove that $\frac{19}{13}$ is not a Yano's candidate while $-\frac{19}{13}$ as it
can be checked with \texttt{checkRoot} in \texttt{Singular} (working over $\bc(t)$ instead of randomly evaluating~$t$).
\begin{figure}[ht]
\begin{center}
\iftikz
\begin{tikzpicture}[xscale=1.5,vertice/.style={draw,circle,fill,minimum size=0.2cm,inner sep=0}]
\coordinate (M0) at (0,0);
\coordinate (M1) at (1,0);

\foreach \i in {2,...,8}
{
\coordinate (M\i) at (\i,0);
}

\foreach \i in {1,...,7}
{
\node[vertice] at (M\i) {};
}

\coordinate (Z) at (0,1);

 \draw[-{[scale=2]>}] (M1)--(M0);
 \draw[-{[scale=2]>}] (M7)--($(M8)+(Z)$);
 \draw[-{[scale=2]>}] (M7)--($(M8)-(Z)$);
 \draw[-{[scale=2]>}] (M2)--($(M2)+(Z)$);
\draw (M1) -- (M7) ;

\node[below] at (M3) {$(5,2)$};
\node[below] at (M1) {$(7,3)$};
\node[below] at (M2) {$(13,5)$};
\node[below] at (M4) {$(7,3)$};
\node[below] at (M5) {$(9,4)$};
\node[below] at (M6) {$(11,5)$};
\node[right=7pt] at (M7) {$(13,6)$};

\end{tikzpicture}
\else
\includegraphics[scale=.5]{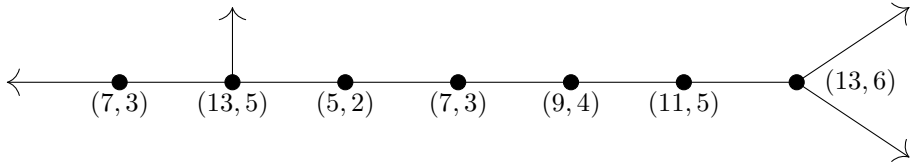}
\fi
\caption{Resolution graph for Example~\ref{ejm:acampo-23-210}}
\end{center}
\end{figure}
\end{ejm}

\begin{ejm}
Let $f(x,y):=y^{10}-x^3 y^{5}-x^{12}$.
A $\mu$-constant  versal deformation is given by
\begin{gather*} 
f_{\mathbf{t}}(x,y):=f(x,y)+t_1 x^7 y^3 + t_2 x y^9  + t_3 x^9 y^2 + t_4 x^8 y^3 + t_5 x^{11} y \\
+ t_6 x^{10} y^2 + t_7 x^9 y^3 + t_8 x^{11} y^2 + t_9 x^{10} y^3 + t_{10} x^{11} y^3.
\end{gather*}
\end{ejm}
Using random values we can prove that $-\frac{19}{15}$ and $-\frac{4}{15}$ are both roots of the Bernstein
polynomial, but only $\frac{4}{15}$ is a Yano's candidate.

% \bibliographystyle{amsplain}
% \bibliography{biblio}

\def\cprime{$'$}
\providecommand{\bysame}{\leavevmode\hbox to3em{\hrulefill}\thinspace}
\providecommand{\MR}{\relax\ifhmode\unskip\space\fi MR }
\providecommand{\MRhref}[2]{%
  \href{http://www.ams.org/mathscinet-getitem?mr=#1}{#2}
}
\providecommand{\href}[2]{#2}

\end{document}